\documentclass[twoside]{article}
\usepackage[T1]{fontenc}
\usepackage[latin1]{inputenc}
\usepackage{amsmath,amsthm}
\usepackage{graphicx}
\usepackage{microtype}

\newtheorem{theo}{Theorem}[section]
\newtheorem{lemme}[theo]{Lemma}

\makeatletter
\@addtoreset{equation}{section}
\makeatother

\newcommand\egaldef{\stackrel{\mbox{\upshape\tiny def}}{=}}
\newcommand{\PP}{\mathrm{I\! P}}
\newcommand{\EE}{\mathrm{I\! E}}
\newcommand{\RR}{\mathrm{I\! R}}
\newcommand{\ZZ}{\mathsf{Z}}

\newcommand{\N}{^{\scriptscriptstyle (N)}}
\DeclareMathOperator{\dilog}{dilog}
\newcommand{\1}{\leavevmode\hbox{\rm \small1\kern-0.35em\normalsize1}}
\newcommand{\ind}[1]{\1_{\{#1\}}}

\newcommand{\EX}{{\bar\alpha}}

\begin{document}

\title{A nonlinear integral operator encountered in the bandwidth
       sharing of a star-shaped network\thanks{This work has been
       partly supported by a grant from the \emph{Centre National d'\'Etudes
       en T\'el\'ecommunications}.}} 
\author{Guy Fayolle\thanks{INRIA -- Domaine de Voluceau BP 105 --
Rocquencourt 78153 Le Chesnay cedex\protect\\
E-mail: \texttt{Guy.Fayolle@inria.fr}, 
\texttt{Jean-Marc.Lasgouttes@inria.fr}}
\and Jean-Marc Lasgouttes\footnotemark[1]}
\date{}
\maketitle

\begin{abstract}
We consider a symmetrical star-shaped network, in which bandwidth is
shared among the active connections according to the ``min'' policy.
Starting from a \emph{chaos propagation} hypothesis, valid when the
system is large enough, one can write equilibrium equations for an
arbitrary link of the network. This paper describes an approach based
on functional analysis of nonlinear integral operators, which allows
to characterize quantitatively the behaviour of the network under
heavy load conditions.

\end{abstract}

\section{Model description}\label{sec.presentation}

\begin{figure}
\begin{center}
\includegraphics[width=0.5\textwidth,keepaspectratio]{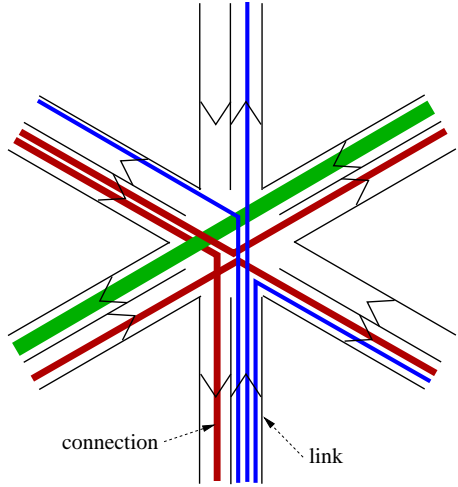}
\end{center}
\caption{The star-shaped network}
\label{fig.etoile}
\end{figure}
Consider a network comprising $N$ links, where several data sources
establish connections along routes going through these links. The main
concern here is about the policies that can be used to share the
bandwidth of the links between active connections, and the effect
of these policies on the dynamics of the network.

In this paper, the network is star-shaped (see
Figure~\ref{fig.etoile}) and all routes are of length $2$, which is a
reasonable model for a router. Each star branch contains two links
(``in'' and ``out'') and each route is isomorphic to a pair of
links.

Let $r=(i,j)$ denote a route on links $i$ and $j$, and $\mathcal{R}$
be the set of all possible routes (with cardinal $N^2/2$). Connections
are created on $r$ according to a Poisson process with intensity
$2\lambda/N$, so that the total arrival intensity on each link is
$\lambda$. A connection lasts until it has transmitted over the
network its data, the volume of which follows an exponential law with
mean $v$. Each link $i,\,1\leq i\leq N$, has a bandwidth equal to $1$
and its load is $\rho\egaldef\lambda v$.

The state of the system at time $t\in\RR$ is given by the number of
active connections on each route $\bigl(c\N_r(t),\,r\in
\mathcal{R}\bigr) \egaldef X\N(t)$. The vector $X\N(t)$ is in general
Markovian and
\[
 X\N_i(t)\egaldef \sum_{r\ni i}c\N_r(t)
\]
is the total number of active connections on link $i,\,\forall 1\leq
i\leq N$.

It is now necessary to describe how bandwidth allocation is achieved.
The ``max-min fairness'' policy, popular in telecommunication models,
being too difficult to be studied rigorously, this article focuses on
the ``min'' policy, proposed by L.~Massoullié and
J.~Roberts~\cite{MasRob:1}, in which a connection on $(i,j)$ gets
bandwidth
\begin{equation}\label{eq.debit}
\min\Bigl(\frac{1}{X\N_i(t)},\frac{1}{X\N_j(t)}\Bigr).
\end{equation}

This allocation clearly satisfies the capacity constraints of the
system, and can be shown to be sub-optimal with respect to max-min
fairness. Its invariant measure equations are however too complicated
to be solved explicitly. An efficient method in such situations is to
study the model in so-called \emph{thermodynamical limit}, using
\emph{mean field} analysis.

In order to study the stationary behaviour of the network as
$N\to\infty$, the idea is to assume the conditions of \emph{chaos
propagation}, under which any finite number of links tend to become
mutually independent. In this paper, this hypothesis will be
considered as a heuristic, to be proved in further studies. Some
rigorous studies of this type can be found in Vvedenskaya \emph{et
al.}~\cite{VveDobKar:1} and Delcoigne et Fayolle~\cite{DelFay:2}.

 From now on, $\rho<1$ and the system is assumed to be in stationary
state $X=(c_r,r\in\mathcal{R})$. For all $k\geq0$, the empirical
measure of the number of links with $k$ connections is
\[
 \alpha\N_k\egaldef\frac{1}{N}\sum_{1\leq i\leq N}\ind{X\N_i = k}.
\]

Symmetry considerations imply that, for all $i\leq N$,
$\PP(X\N_i=k)=\EE\alpha\N_k$. Besides, from the chaos propagation
hypothesis, a law of large numbers is assumed to hold for $\alpha\N_k$:
\[
 \alpha_k \egaldef \lim_{N\to\infty}\alpha\N_k 
 = \lim_{N\to\infty}\PP(X\N_i=k) \egaldef \PP(X=k).  
\]

The $\alpha\N_k$'s, traditionally named \emph{mean field}, drives the
dynamics of the system. The following notation will also prove useful:
\[
  \EX\N \egaldef \EE X\N_i = \sum_{k>0}k\alpha\N_k,\qquad
  \EX   \egaldef \sum_{k>0}k\alpha_k.
\]

A heuristic computation (detailed in~\cite{DenFayForLas:1}), yields
the following equations:
\begin{eqnarray}
u_k &=&\sum_{\ell>0}(k\wedge\ell)\alpha_l , \nonumber\\
\alpha_{k+1} u_{k+1} 
  &=& \rho \EX\alpha_k, \qquad \forall k\geq0.\label{invariante_EQ}
\end{eqnarray}

While (\ref{invariante_EQ}) resembles a ``birth and death process''
equation, it is in fact highly non-linear, due to the form of $u_k$
and of $\EX$. 

The purpose of this paper is to show how the asymptotic behaviour of
the system (as $\rho\to1$) can be derived from the analytical study of the
generating functions built from~(\ref{invariante_EQ}). This work is a part
of the wider study~\cite{DenFayForLas:1}, which also gives ergodicity
conditions for any topology under the min and max-min policies, shows
how equations like~(\ref{invariante_EQ}) are derived (also in the case
where routes are longer than $2$) and presents comprehensive numerical
results.

The main byproduct of Theorem~\ref{theo-anal7}, is the following
asymptotic expansions, valid in a neighborhood of $\rho=1$.
\begin{eqnarray*}
 \EX&\approx&\frac{1}{(1-\rho)^2A},\\
 \lim_{k\to\infty}\rho^{-k}\alpha_k 
       &\approx& (1-\rho)B\exp\Bigl[\frac{1}{(1-\rho)A}\Bigr],
\end{eqnarray*}
where $A$ and $B$ are non-negative constants. Moreover, if $c(z,1)$
and $v(z,1)$ are the solutions of the system of differential
equations~(\ref{eq-anal14}), then $A$ can be written as follows:
\[
 A=\int_0^\infty c(z,1)dz=\lim_{z\to\infty}zv'(z,1)\approx 1.30.
\]

Since this system is numerically highly unstable, it has proven
difficult (with the ``Livermore stiff ODE'' solver from MAPLE) to get
a better estimate for $A$.

\section{An integral equation for the generating function}

Let $\mathcal{C}(r)$ (resp.\ $\mathcal{D}(r)$), be the circle (resp.\
the open disk) of radius $r$ in the complex plane.

Let $\alpha : z\rightarrow\alpha(z)$ be the generating function,
\emph{a priori} defined for $z$ in the closed unit disk
\[ 
  \alpha(z) \egaldef \sum_{k\geq 0}\alpha_k z^k .
\] 

Denoting $\alpha'$ the derivative of $\alpha$,~(\ref{invariante_EQ})
can be rewritten as
\begin{equation}\label{eq-anal0}
\alpha_{k+1} u_{k+1} = \rho \alpha'(1)\alpha_k, \quad \forall k\geq 0.
\end{equation}

\begin{lemme}\label{lem-anal1} \mbox{ }
\begin{itemize}
\item[(a)] If (\ref{eq-anal0}) has a probabilistic solution, then,
necessarily $\rho<1$, and
\[
 \lim_{k\rightarrow\infty}\alpha_k\rho^{-k} = K(\rho), 
\]
where $K(\rho)$ is a positive constant, bounded $\forall\rho<1$.
\item[(b)] The function $\alpha$ satisfies the nonlinear integral equation
\begin{equation} \label{eq-anal1}
\alpha '(1) (1-\rho z)\alpha(z) =
\dfrac{1}{2i\pi}\int_{\mathcal{C}(r)} \alpha(\omega)\,\alpha
\Bigl(\dfrac{z}{\omega}\Bigr)\dfrac{d\omega}{(1-\omega)^2} \, , 
\end{equation}
where $|z|<\rho^{-1}$ et $r$ is an arbitrary positive number, with 
$1<r<\rho^{-1}$.
\end{itemize}
\end{lemme}

\begin{proof}{}
When~(\ref{invariante_EQ}) has a probabilistic solution, necessarily
$\EX=\lim_{m\rightarrow\infty} u_m$. Therefore, $\forall\epsilon>0$,
there exists a number $M(\epsilon)>0$, such that

\[ 
  \rho\leq\dfrac{\alpha_{m+1}}{\alpha_m} \leq \dfrac{\rho}{1-\epsilon},
  \quad \forall m \geq M(\epsilon),
\] 
which implies that $\rho<1$. Moreover, under the same existence
hypothesis, one can write
\begin{equation}\label{eq-anal2}
\alpha_k\rho^{-k} = \alpha_0 \prod_{l\geq 0}^k
\dfrac{\EX}{\EX - D_{\ell+1}},
\end{equation}
where $D_\ell = \sum_{m\geq \ell}(m-\ell)\alpha_m$.

When $\ell\to\infty$, the convergence of the product
in~(\ref{eq-anal2}) is equivalent to the convergence of the series of
general term $D_\ell$, which holds since
\[  
  D_\ell \leq \alpha_0 \sum_{m\geq \ell}(m-\ell)(\rho +\epsilon)^m , 
    \ \forall \ell \geq M(\epsilon) , 
\] 
so that $\sum_\ell D_\ell$ behaves like
\[
  \sum_{\ell\geq 0}\sum_{n\geq 0} n(\rho +\epsilon)^{\ell+n} =
  \dfrac{\rho+\epsilon}{(1-\rho-\epsilon)^3}. 
\] 

Point (a) of the lemma is proven. Point (b) is an application of (a)
and of an integral representation used by Hadamard, recalled below
(see e.g.~\cite{TIT}).
\begin{itemize}\item[]\itshape%
Let $a$ and $b$ be functions
\[
  a(z) = \sum_{n\geq 0} a_n z^n, \qquad b(z) = \sum_{n\geq 0} b_n z^n ,
\]
analytic in the respective disks $\mathcal{D}(R)$ and
$\mathcal{D}(R')$. Then the function
\[
 c(z)= \sum_n a_n b_n z^n
\] 
has a radius of convergence greater than $RR'$ and has the integral form
\[
 c(z) = \dfrac{1}{2i\pi}\int_{\mathcal{L}} a(\omega)\,
  b\Bigl(\dfrac{z}{\omega}\Bigr)\dfrac{d\omega}{\omega} \, ,
\]
where $\mathcal{L}$ is a closed contour containing the origin, and on
which $|\omega|<R, \Bigl|\dfrac{z}{\omega}\Bigr|<R'$.
\end{itemize}

 From (a), $\alpha$ has a radius of convergence at least equal to
$\rho^{-1}$: this property, used in Hadamard's formula, leads directly
to~(\ref{eq-anal1}). 
\end{proof}

The following Lemma provides a finer description of $\alpha(z)$.

\begin{lemme}\label{lem-anal2} 
The function $\alpha$ is meromorphic and can be written as
\begin{equation}\label{eq-anal3}
\alpha(z) = \sum_{i=1}^{\infty} \dfrac{a_i}{\rho^{-i} - z} \,.
\end{equation}

Moreover, for any sequence of circles $\mathcal{C}(R_n)$, such that
\[
 (1+\epsilon)\rho^{-n}\leq R_n \leq (1-\epsilon)\rho^{-(n+1)}, \quad
 0<\epsilon<\dfrac{1-\rho}{1+\rho}, 
\] 
one has $|\alpha(R_n)| = o(1)$ as $n\rightarrow\infty$.
\end{lemme}

\begin{proof}{}
The integral equation~(\ref{eq-anal1}) allows for the analytic
continuation of $\alpha$ in the whole complex plane. Indeed, from
point (a) of Lemma~\ref{lem-anal1}, $\alpha$ is holomorphic in
$\mathcal{D}(\rho^{-1})$ and its first singularity is a simple pole at
the point $z=\rho^{-1}$. An application of Cauchy's theorem to the
integral in~(\ref{eq-anal1}) leads to
\[
  \dfrac{1}{2i\pi}\int_{\mathcal{C}(R_1)} =
 \dfrac{1}{2i\pi}\int_{\mathcal{C}(r)} \ + \ \mbox{Residue}(\rho^{-1}),
\]
where $R_1$ is defined above. Since $\alpha(\omega)$ and
$\alpha(z/\omega)$ are analytic in the regions
\[ 
  |\omega | < \rho^{-1}, \quad
   \Bigl|\dfrac{z}{\omega}\Bigr|<\rho^{-1}, 
\] 
the integral is a function of $z$ analytic in the ring-shaped area 
$\rho^{-1}<|z|<\rho^{-2}$. The same holds for $\alpha$, thanks to the
left-hand side of~(\ref{eq-anal1}). Equation~(\ref{eq-anal3}) follows by
recurrence. 

The second part of the lemma is obtained by application of Cauchy's
theorem to the integral in~(\ref{eq-anal1}) along the circle
$\mathcal{C}(R_n)$:
\[ 
 \alpha '(1) (1-\rho z)\alpha(z) - \sum_{i=1}^n
   \dfrac{a_i\alpha(z\rho^i)}{(\rho^{-i} - 1)^2} =
   \int_{\mathcal{C}(R_n)}\alpha(\omega)\,\alpha\Bigl(\dfrac{z}{\omega}\Bigr)
   \dfrac{d\omega}{(1-\omega)^2}.
\]

One can make the analytic continuation of the above equality, the
left-hand side of which is analytic in $|z|\leq R_n$, letting $z$ reach
the circle $\mathcal{C}(R_n)$ along a simple curve which avoids the
polar singularities $\rho^{-i}, i=1,\ldots,n$. Bounding the modulus
of the integral yields the inequality
\[ 
  (\rho R_n -1)\mathcal{M}(R_n) \leq A_n \sup_{i\leq n-1}
   (\mathcal{M}(R_i)) + D\mathcal{M}(R_n)\int_{\mathcal{C}(R_n)}
   \dfrac{|d\omega|}{|1-\omega|^2},
\] 
where $D$ is a positive constant and $A_n$ is bounded $\forall
n\leq\infty$. By induction, $\lim_{n\to\infty}\mathcal{M}(R_n)=0$
follows easily, and the proof is completed.
\end{proof}

Combining~(\ref{eq-anal3}) with~(\ref{eq-anal1}), a straightforward
computation of residues gives
\begin{equation}\label{eq-anal4}
 \alpha '(1)(1-\rho z) 
     \sum_{i=1}^{\infty} \dfrac{a_i}{\rho^{-i} - z} 
 = \sum_{i,j=1}^{\infty} \dfrac{a_i a_j\rho^{-j}}
                               {(\rho^{-(i+j)}-z)(1-\rho^{-j})^2}.
\end{equation}

Let $a : z\to a(z)$ be the generating function
\[
  a(z) \egaldef \sum_{k\geq 0}a_{k+1} z^k ,
\]
defined for $z$ in a bounded domain of the complex plane,
including the origin. Using point (a) of Lemma~\ref{lem-anal1}
and~(\ref{eq-anal3}), the following relations hold
\begin{equation}\label{eq-anal5}
\begin{cases}
a_1 = \rho K(\rho)\\[0.2cm]
\alpha_{k-1} = \rho^k a(\rho^k), \quad \forall k\geq 1,
\end{cases}
\end{equation}

Since $\alpha_0 <1-\rho$, the function $a$ is thus analytic in the
disk $\mathcal{D}(\rho)$. Identifying the coefficients of the power
series in $z$ in~(\ref{eq-anal4}), one gets
\begin{equation}\label{eq-anal6}
\alpha '(1)(\rho^{k-1}-1)a_k =
\sum_{j=1}^{k-1}\dfrac{a_j a_{k-j}\rho^{j+k-1}}{(1-\rho^j)^2}, \quad
\forall k\geq 1.
\end{equation}
It follows easily by recurrence that the $a_i$'s are of
alternate signs, with $a_1>0$. 

Let
\begin{equation}\label{eq-anal7}
  f(t) \egaldef \sum_{j=1}^{\infty}\Bigl(\dfrac{\rho^j}{1-\rho^j}\Bigr)^2t^j 
     = t\sum_{j=1}^{\infty}\dfrac{j\rho^{j+1}}{1-t\rho^{j+1}}, 
 \quad \rho<1,\, |t|<\rho^{-2}.
\end{equation}

Hadamard's formula, when applied in~(\ref{eq-anal6}), implies
the integro-differential equation
\begin{equation}\label{eq-anal8}
\alpha '(1)[a(\rho z) - a(z)] 
  = \dfrac{a(\rho z)}{2i\pi} 
    \int_{\mathcal{C}(r)} a(\omega)f\Bigl(\dfrac{z}{\omega}\Bigr)d\omega \,,
\end{equation}
valid in the domain $\bigl\{r\leq\rho,\, |z|<\rho^{-2}\bigr\}$.

Taking the second form for $f$ in~(\ref{eq-anal7}), which in fact
converges in the domain $\bigl\{\rho<1,\,\Re(z)\leq 0\bigr\}$, and
applying Cauchy's theorem to the integral in~(\ref{eq-anal8}), one
obtains the functional equation
\begin{equation}\label{eq-anal9}
\alpha '(1)\bigl[a(\rho z) - a(z)\bigr] 
   = a(\rho z)b(z) , \quad |z|\leq \rho,\ \rho <1,
\end{equation}
where
\[
  b(z) \egaldef z\sum_{j=1}^\infty j\rho^{j+1}\, a(\rho^{j+1}z).
\] 

 From~(\ref{eq-anal9}), let's make now the analytic continuation of
$a$ in the nested disks $\mathcal{D}(\rho^{-n}),n\geq 1$. It appears
that $a$ has no singularity at finite distance, and consequently is an
\emph{integral} function. From the general theory~\cite{TIT}, it
follows that $a$ is completely characterized by its zeros and its
\emph{order} at infinity.

Let $z_0$ be an arbitrary zero of $a$. From~(\ref{eq-anal9}) again,
$a(\rho^{-i}z_0)=0$ and the zeros of $a$ form families of points in
geometric progression with parameter $\rho^{-1}$. It suffices to
determine the zeros of smallest modulus, but, alas they do not have
any explicit form and numerical schemes are highly unstable. However,
from~(\ref{eq-anal5}), $b(1)=\alpha'(1)$, so that~(\ref{eq-anal9})
implies $a(1)=0$ together with
\begin{equation}\label{eq-anal9,5}
  a(\rho^{-i})=0, \quad\forall i\geq 0.
\end{equation}

\section{On the asymptotic behaviour around $\rho=1$}

In order to assess the practical value of the ``min'' policy, it is
important to evaluate the system behaviour in heavy traffic
conditions. The numerical calculations in~\cite{DenFayForLas:1} show
that the distribution of any $X_i$ is \emph{modal}, which is not
common in known models.

In this section, it will be convenient to consider $\rho$ not only as
a parameter but as a plain variable. Therefore, in all quantities of
interest, $\rho$ will appear as an explicit variable, e.g.\
$a(z,\rho)$, $f(z,\rho)$ or $a_k(\rho)$.

The fundamental ideas of the analysis will be given after the next
lemma, which proposes a scaling---likely to be the only interesting
one---for the function $a(z,\rho)$.

Let
\begin{equation}\label{eq-anal10} 
\begin{cases}
\xi(\rho) \egaldef 
   \dfrac{-a_1(\rho)}{(\log\rho)^3\alpha'(1,\rho)},\\[0.3cm] 
c_k(\rho) \egaldef
   \dfrac{-a_k(\rho)}{(\log\rho)^3 \alpha'(1,\rho)\xi(\rho)^k}, \quad
   \forall k\geq 1 ,\\[0.3cm] 
c(z,\rho) \egaldef 
   \displaystyle\sum_{i=0}^{\infty} c_{k+1}(\rho)z^k.
\end{cases}
\end{equation}

The reader will easily convince himself that the factor $(1-\rho)^3$
arises rather naturally; however the factor $-\log^3\rho$ has been
chosen here, since it provides more compact formulas in the
forthcoming results.

\begin{lemme}\label{lem-anal3}
Let
\[ 
v(z,\rho)\egaldef \frac{(\log\rho)^2}{2i\pi} \int_{\mathcal{C}(1)}
 c(\omega,\rho)f\Bigl(\dfrac{z}{\omega},\rho\Bigr)d\omega,\quad
 |z|<\rho^{-2}.
\]

The following functional relations hold:
\begin{equation}\label{eq-anal11} 
\begin{cases}
a(z,\rho) = a_1(\rho)\, c(z\xi(\rho),\rho), \\ c(0,\rho)=1, \\
c(z,\rho) = c(\rho z,\rho)\bigl[1+\log\rho \, v(z,\rho)\bigr] ,\\[0.1cm]
\hspace{1.3cm} = c(\rho z,\rho)\bigl[1+ z(\log\rho)^3 \displaystyle
\sum_{i=1}^{\infty} i\rho^{i+1} c(\rho^{i+1}z,\rho)\bigr].
\end{cases}
\end{equation}

The coefficients $c_k(\rho),\,k\geq1$, are of alternate signs and the
function $c(z,\rho)$ has the following properties.

(a) there exists only one solution $c(z,\rho)$ of~(\ref{eq-anal11}),
which is integral with respect to $z$ and bi-analytic in
$(z,\rho)$ in the region $0<\rho<1$.

(b) Define 
\[
  g(t) \egaldef \sum_{j=1}^{\infty} \dfrac{t^{j}}{j^2} =
    \int_0^t\dfrac{-\log (1-u)du}{u}, \quad\forall |t| \leq 1,
\] 
denoted by some authors as $\dilog(1-t)$. Then $c(z,1)$ exists and
satisfies the integro-differential equation
\begin{equation}\label{eq-anal12}
z\dfrac{\partial c(z,1)}{\partial z} = -\dfrac{c(z,1)}{2i\pi}
\int_{\mathcal{C}(r)}
c(\omega,1)g\Bigl(\dfrac{z}{\omega}\Bigr)d\omega, \quad |r|= 1, \
\forall |z|\leq 1,
\end{equation}
which rewrites in the form
\begin{equation}\label{eq-anal13}
z\dfrac{\partial c(z,1)}{\partial z} = - c(z,1)\int_0^z
c(\omega,1)\log \left(\dfrac{z}{\omega}\right)d\omega ,
\end{equation}
which is equivalent to the non-linear differential system
\begin{equation}\label{eq-anal14}
\begin{cases}
z\dfrac{\partial c(z,1)}{\partial z} + c(z,1) v(z,1) =0,\\[0.3cm]
z\dfrac{\partial^2v(z,1)}{\partial z^2} +
\dfrac{\partial v(z,1)}{\partial z} = c(z,1),
\end{cases}
\end{equation}
with initial conditions
\[
  v(0,1) = 0, \quad \dfrac{\partial v(z,1)}{\partial z}_{|z=0} =1 ,
  \quad c(0,1) =1 .
\]

(c) Moreover, $c(z,1)$ is analytic in the open complex plane, except
at a negative real point $q$, and $c(z,1) \neq 0,\ \forall z\neq q\cup \infty$.
\end{lemme}

\begin{proof}
The first three equations in~(\ref{eq-anal11}) follow directly from
the definition~(\ref{eq-anal6}) and~(\ref{eq-anal9}) of the
coefficients $a_k(\rho)$, the fourth one coming from the analytic
continuation of~(\ref{eq-anal8}).  Existence and uniqueness are simple
consequences of the convolution equation~(\ref{eq-anal6}).

The properties relative to the morphology of $c(z,\rho)$, $\rho\leq
1$, are more intricate. First, the reader will notice
that~(\ref{eq-anal12}) can be obtained rigorously
from~(\ref{eq-anal6}) or~(\ref{eq-anal7}), but not
from~(\ref{eq-anal11})! Then, there is a \emph{phase transition} when
$\rho=1$. We will return to this topic in Section~\ref{Remarques}.
\end{proof}

It is interesting to note that the function $w(y)\egaldef v(e^y)+1$
satisfies the so-called \emph{Blasius}~\cite{Bla:1} third-order
differential equation
\[
 w'''(y)+w(y)w''(y)=0,
\]
which arises in hydrodynamics to describe the stationary evolution of
a laminar boundary layer along a flat plate! The explicit solution of
this equation is still unknown, albeit it has been studied by many
authors over the last decades (see e.g.~\cite{Sch:1}).

\medskip
Starting from Lemma~\ref{lem-anal3}, it is now possible to sketch the
main ideas of the proposed method. The Gordian knot amounts to the
evaluation of $\xi(\rho)$, defined in~(\ref{eq-anal10}). This can be
done via the \emph{anchoring equation}
\begin{equation}\label{eq-anal16} 
c(\xi(\rho),\rho) = 0,
\end{equation}
which follows from~\ref{eq-anal9,5} and from the first equation
of~(\ref{eq-anal11}). From the structure of the third equation
of~(\ref{eq-anal11}), it appears that the smallest positive solution
of $c(u,\rho)=0$ satisfies
\[
  \dfrac{-1}{u\log\rho} = (\log\rho)^2\sum_{i\geq
  1}i\rho^{i+1}c(u\rho^{i+1},\rho),
\]
the right-hand side of which is an analytic function, bounded in any
compact set $\forall\rho\leq 1$: as $\rho\to1$, necessarily
$u\to\infty$ and all positive zeros of the anchoring equation are sent
to infinity. The key is to find the asymptotic behaviour in $z$ of the
various functions, in the cone $0\leq z\leq\mathcal{U}(\rho)$, which
contains $\xi(\rho)$: in this cone, $c(z,\rho)$ is close to
$c(z,1)$---in some sense---and
\begin{equation}\label{eq-anal17}
v(z,\rho)\approx w(z,\rho)\egaldef\log^2\rho\sum_{i\geq 1}
i\rho^{i+1}c(z\rho^{i+1},1).
\end{equation}

Since $w(z,1)$ has a logarithmic behaviour, $\xi(\rho)$ can be
obtained by direct inversion. The sketch of the proof is outlined
below.

We will need the Mellin transform (see e.g.~\cite{Doe}) of $c(z,1)$,
defined as
\[
  c^*(s) \egaldef \int^{\infty}_0 x^{s-1}c(x)dx. 
\] 

The behaviour of $c(z,1)$ in the region $\Re(z)> 0$, given in
Section~\ref{Remarques}, implies the existence of $c^*(s,1), \forall
s,\,\Re(s)>0$ and of all the moments of $c(z,1)$ on the positive real
axis.

\begin{lemme}\label{lem-anal5}
The function $w(z,\rho)$ defined by~(\ref{eq-anal17}) admits, $\forall
z,\, \Re(z)>0$, the asymptotic expansion
\begin{equation} \label{eq-anal30}
\begin{split}
 w(z,\rho) &= c^*(1,1)\log
 (\rho z) - \dfrac{\partial c^*(1,1)}{\partial s}  \\ &+
 (\log\rho)^2\bigl[(\log z)\Phi(z,\rho) +\Psi(z,\rho)\bigr] +
 \mathcal{O}(z^{-d}),
\end{split}
\end{equation}
where $d$ is an arbitrary positive number, $\Phi$ and $\Psi$ are
fluctuating functions of small amplitude for $\rho\neq1$, which vanish
for $\rho=1$.
\end{lemme}

\begin{proof} 
 From the Mellin transform inversion formulas,
\begin{equation}\label{eq-anal31}
 w(z,\rho) = \frac{1}{2i\pi} \int_{\sigma-i\infty}^{\sigma+i\infty}
              z^{-s}c^*(s+1,1)\Bigl[\frac{\log\rho}{1-\rho^s}\Bigr]^2ds,
              \quad \forall \sigma \in (-1,0).
\end{equation}

Let $s_n\egaldef2in\pi\log^{-1}\rho$, $\forall n\in\ZZ$. Cauchy's
theorem can be applied to~(\ref{eq-anal31}), by integrating along the
vertical line $\Re(s)=d>0$, so that
\begin{equation}\label{eq-anal32}
w(z,\rho) = \sum_{n\in\ZZ} z^{-s_n}\Bigl[c^*(s_n+1,1)\log (\rho z) -
\dfrac{\partial c^*(s_n+1,1)}{\partial s}\Bigr] +\mathcal{O}(z^{-d}).
\end{equation}
The series above is equal to the sum of the residues, taken on the
vertical line $\Re(s)=0$, and is uniformly bounded,
$\forall\rho\leq 1$. Indeed, an integration by parts gives the
inequality
\[
 \bigl|c^*(s_n+1,1)\bigr| = \Bigl|\int_0^\infty y^{s_n}c(y,1)dy\Bigr| \leq
   \left|\frac{\Gamma(s_n+1)}{\Gamma(s_n+k+1)}\right| \int_0^\infty
   y^k|c^{(k)}(y,1)|dy,
\]
where $c^{(k)}$ is the $k$-th derivative of $c(z,1)$ and
$\Gamma (x)$ is the usual Eulerian function. Since
\[
  \left|\frac{\Gamma(s_n+1)}{\Gamma(s_n+k+1)}\right| 
  < \frac{\log^k\rho}{k!},
\] 
the proof of~(\ref{eq-anal30}) and of the lemma is concluded.
\end{proof}

\begin{theo}\label{theo-anal7}
For some real number $d>1$, the following expansions hold.
\begin{eqnarray}
\rho\xi(\rho) 
  &=& \exp\biggl[-\frac{1}{\log\rho\, c^*(1,1)}
                -\frac{\partial \log c^*(1,1)}
                                      {\partial s}\biggr]
      \bigl(1+\mathcal{O}((\log\rho)^d)\bigr),\label{eq.devxi}\\[0.5em]
\alpha'(1,\rho) 
  &=& \frac{1}{(\log\rho)^2 c^*(1,1)}+\mathcal{O}((\log\rho)^d).
							\label{eq.devalphap}
\end{eqnarray}
\end{theo}

\begin{proof}
 From~(\ref{eq-anal16}), $\xi(\rho)$ is solution of the equation in $x$
\[
 1+\log\rho\, v(x,\rho) = 0.
\]
A deep analysis, which is not included here, shows that this
equation can be replaced by the locally equivalent equation
\begin{equation}\label{eq-anal18}
 1+\log\rho\, w(x,\rho) + \mathcal{O}\bigl((\log\rho)^p\bigr)= 0,
\end{equation}
where $p$ is a positive number, $p> 1$. Then, from Lemma~\ref{lem-anal5},
\begin{equation}\label{eq.calcxi}
 -\frac{1}{\log\rho} 
   = \log(\rho x) c^*(1,1)
     -\frac{\partial c^*(1,1)}{\partial s}
     +\log(\rho x)\Phi( x,\rho)
     +\Psi(x,\rho)+\mathcal{O}(x^{-d}),
\end{equation}
which implies~(\ref{eq.devxi}). 

For the computation of $\alpha'(1,\rho)$, one uses the simple relation
\[
 \rho\alpha'(1,\rho) = \frac{\alpha_1}{\alpha_0}-\alpha_1,
\]
obtained by derivation of~(\ref{eq-anal1}) at $z=0$,
and~(\ref{eq.devalphap}) follows.
\end{proof}

\section{Remarks and complements}\label{Remarques}
It is important to note that one of the main technical difficulties of
the problem, besides its strongly non-linear feature, comes from the
\emph{phase transition} which appears for $\rho=1$. Actually,
$c(z,\rho)>0$ for $0\leq z<\xi(\rho)$ and $c(\xi(\rho),\rho)=0$. Then,
for $z\gg\xi(\rho)$, $c(z,\rho)$ has wild unbounded oscillations. In
particular, this implies that $\int_0^{\infty}c(x,\rho)dx$ does not
exist.

On the other hand, when $\rho=1$, the $c(z,1)$ is no more an integral
function: it has a singularity (which seems to be a pole of order $3$)
located on the negative real axis; it does not vanish for $z\geq 0$,
and the quantity $\int_0^{\infty}c(x,1)dx=A$ is finite. In the half
plane $\Re(z) >0$, the following expansions hold:
\begin{eqnarray*}
  c(z,1) &=& 
   \exp \left[-\dfrac{c^*(1,1)}{2}\log^2 z + B\log z
              +\dfrac{D\log z}{z^2} 
              +o\Bigl(\dfrac{\log z}{z^2}\Bigr)\right],\\
  v(z,1) &=& 
   c^*(1,1)\log z + B + \dfrac{D\log z}{z^2} 
   +o\biggl(\dfrac{\log z}{z^3}\biggr),
\end{eqnarray*}
where $B$ and $D$ are some constants.

\bigskip
Finally, iterating~(\ref{eq-anal11}), one could improve some of the
estimates given in the previous section, rewriting $c(z,\rho)$ as
\[ 
c(z,\rho) = 
  c(z\rho^{I+1},\rho)\prod_{i=0}^I
                       \bigl[1+\log\rho \,v(z\rho^i,\rho)\bigr],
\]
where $I$ is an arbitrary positive integer. The above product is
uniformly convergent, $\forall I\leq\infty$, for all $z$ in a compact
set of the complex plane, since it behaves like the series
\[
z(\log\rho)^3\sum_{k\geq 0}\sum_{i\geq 1}
i\rho^{i+k+1}c(z\rho^{i+k+1},\rho).
\] 

This series has its modulus bounded by $|z\rho^2c(\rho z,\rho)|$, so
that, from the maximum modulus principle, it converges uniformly.
Then, one can choose $I$ to ensure
\[
  0<z\rho^I\leq 1, 
 \mbox{ i.e.\ } I\leq - \dfrac{\log z}{\log\rho},
\]
and make use of the properties of the $\Gamma$ function to estimate
the product.


\end{document}